\documentclass[10pt]{amsart}
\usepackage{amsfonts}
\usepackage{ifthen}
\usepackage{amsthm}
\usepackage{amsmath,mathrsfs}
\usepackage{graphicx}
\usepackage{amscd,amssymb,amsthm}
\usepackage{color}
\usepackage{hyperref}

\newcounter{minutes}
\divide\time by 60
\newcounter{hours}
\multiply\time by 60 \addtocounter{minutes}{-\time}

\newtheorem{theorem}{Theorem}
\newtheorem{lemma}{Lemma}
\newtheorem{define}{Definition}
\newtheorem{coro}{Corollary}
\newtheorem{remark}{Remark}

\title[$(p,q)-$Mathieu type power series ]{On a new $(p, q)$-MATHIEU–TYPE POWER
SERIES AND ITS APPLICATIONS}

\author[K. Mehrez,  Z. Tomovski]{Khaled Mehrez and Zivorad Tomovski}
\address{D\'epartement de Math\'ematiques Facult\'e des Sciences de Tunis, Universit\'e Tunis El Manar, Tunisia}
\address{D\'epartement de Math\'ematiques, ISSAT Kasserine, Universit\'e de Kairouan, Tunisia}
\email{k.mehrez@yahoo.fr}
\address{\v{Z}ivorad Tomovski. University "St. Cyril and Methodius", Faculty
of Natural Sciences and Mathematics, Institute of Mathematics, Repubic of
Macedonia.}
\email{tomovski@pmf.ukim.edu.mk}
\email{k.mehrez@yahoo.fr}



\begin{document}

\def\thefootnote{}
\footnotetext{ \texttt{File:~\jobname .tex,
          printed: \number\year-0\number\month-\number\day,
          \thehours.\ifnum\theminutes<10{0}\fi\theminutes}
} \makeatletter\def\thefootnote{\@arabic\c@footnote}\makeatother

\maketitle

\begin{abstract}
Our aim in this paper,  is to establish certain new integrals for the  the $(p,q)-$Mathieu--power series. In particular, we investigate the Mellin-Barnes type integral representations
for a particular case of thus special function. Moreover, we introduce the notion of the $(p,q)-$Mittag-Leffler functions and we present a relationships between thus two functions. Some other applications are proved, in particular two Tur\'an type inequalities for the $(p,q)-$Mathieu series are proved.
\end{abstract}

\section{\bf Introduction}
\setcounter{equation}{0}
The following familiar infinite series
\begin{equation}
S(r)=\sum_{n=1}^\infty\frac{2n}{(n^2+r^2)^2},
\end{equation}
is called a Mathieu series. It was introduced and studied by \'Emile Leonard Mathieu in his
book \cite{18} devoted to the elasticity of solid bodies.  Bounds for this series
are needed for the solution of boundary value problems for the biharmonic equations in a two--dimensional rectangular domain, see \cite[Eq. (54), p. 258]{13}.

Several interesting problems and solutions dealing with integral representations and
bounds for the following slight generalization of the Mathieu series with a fractional
power:
\begin{equation}\label{0t1}
S_\mu(r)=\sum_{n=1}^\infty\frac{2n}{(n^2+r^2)^{\mu+1}},\;(\mu>0,\;r>0),
\end{equation}
can be found in the recent works by Diananda \cite{D}, Tomovski and Tren\v{c}evski \cite{TT}, Srivatava et al. \cite{SKZ}. In \cite{SKZ}, the authors derived the following new Laplace type integral representation  via Schlomilch series:
\begin{equation}
S_\mu(r)=\frac{\sqrt{\pi}}{2^{\mu-\frac{1}{2}}\Gamma(\mu+1)}\int_0^\infty e^{-rt} \mathcal{K}_\mu(t)dt, \;\mu>\frac{3}{2},
\end{equation}
where 
$$\mathcal{K}_\mu(t)=t^{\mu+\frac{1}{2}}\sum_{k=1}^\infty\frac{J_{\mu+\frac{1}{2}}(kt)}{k^{\mu-\frac{1}{2}}}$$
with $J_\mu(z)$ is the Bessel function. Motivated essentially by the works of Cerone and Lenard  \cite{C}, Srivastava and Tomovski in \cite{ZY}  defined a family of generalized Mathieu series
\begin{equation}\label{;,}
S_\mu^{(\alpha,\beta)}(r; \textbf{a})=S_\mu^{(\alpha,\beta)}(r; \{a_k\}_{k=0}^\infty)=\sum_{k=1}^\infty\frac{2a_k^\beta}{(a_k^\alpha+r^2)^\mu},\;(\alpha,\beta,\mu,r>0),
\end{equation}
where it is tacitly assumed that the positive sequence
$$\textbf{a}=\{a_k\}=\{a_1,a_2,...\},\;\textrm{such \;that\;}\lim_{k\longrightarrow\infty}a_k=\infty,$$
is so chosen that the in?nite series in de?nition (\ref{;,}) converges, that is, that the following
auxiliary series
$$\sum_{k=1}^\infty \frac{1}{a_k^{\mu\alpha-\beta}},$$
is  convergent.

\begin{define} $($see \cite[Eq. (6.1), p. 256]{SR}$)$ The extended beta function $B_{p,q}(x,y)$  is defined by
\begin{equation}\label{1}
B_{p,q}(x,y)=\int_0^1 t^{x-1}(1-t)^{y-1}E_{p,q}(t)dt,\;x,y,p,q\in\mathbb{C},\Re(p),\Re(q)>0,
\end{equation}
where $E_{p,q}(t)$ is defined by
$$E_{p,q}(t)=\exp\left(-\frac{p}{t}-\frac{q}{1-t}\right),\;p,q\in\mathbb{C},\Re(p),\Re(q)>0.$$
\end{define}
In particular,  Chaudhry et al. \cite[p. 20, Eq. (1.7)]{CH}, introduced the $p–$extension of the Eulerian Beta function $B(x, y):$
$$B_p(x,y)=\int_0^1t^{x-1}(1-t)^{y-1}e^{-\frac{p}{t(1-t)}}dt,\;\Re(p)>0,$$
whose special case when $p=0$ ( or $p=q=0$ in (\ref{1}) )we get the familiar beta integral
\begin{equation}
B(x,y)=\int_0^1 t^{x-1}(1-t)^{y-1}dt,\;\Re(x),\Re(y)>0.
\end{equation}

\begin{define} $($see \cite[p. 4, Eq. 2.1]{23}$)$ Assume that $\lambda,\mu,s,p,q\in\mathbb{C}$ such that $\Re(p),\Re(q)\geq0$ and $\nu,a\in\mathbb{C}\setminus\mathbb{Z}_0^-.$ The extended Hurwitz-Lerch zeta function is defined by
\begin{equation}\label{2}
\Phi_{\lambda,\mu,\nu}(z,s,a;p,q)=\sum_{n=0}^\infty\frac{(\lambda)_n}{n!}\frac{B_{p,q}(\mu+n,\nu-\mu)}{B(\mu,\nu-\mu)}\frac{z^n}{(a+n)^s},\:(|z|<1),
\end{equation}
where $(\lambda)_n$ denotes the Pochhammer symbol (or the shifted factorial) defined, in terms of Euler's Gamma function, by
\begin{displaymath}
(\lambda)_\mu=\frac{\Gamma(\lambda+\mu)}{\Gamma(\lambda)}=\left\{ \begin{array}{ll}
1& \textrm{$(\mu=0;\lambda\in\mathbb{C}\setminus\{0\})$}\\
\lambda(\lambda+1)...(\lambda+n-1)& \textrm{$(\mu=n\in\mathbb{N};\lambda\in\mathbb{C})$}
\end{array} \right.
\end{displaymath}
\end{define}
Upon setting $\lambda=1,$ (\ref{2}) reduces to 
$$\Phi_{\mu,\nu}(z,s,a;p,q)=\sum_{n=0}^\infty\frac{B_{p,q}(\mu+n,\nu-\mu)}{B(\mu,\nu-\mu)}\frac{z^n}{(a+n)^s},\:(|z|<1).$$
 It is easy to observe that 
\begin{equation}\label{ZZ}
\Phi_{\lambda,\mu,\nu}(z,s,a;p,q)=\frac{1}{\Gamma(\lambda)}D_z^{\lambda-1}\{z^{\lambda-1}\Phi_{\mu,\nu}(z,s,a;p,q)\},\;(\Re(\lambda)>0),
\end{equation}
where $D_z^{\lambda}$ denotes the well-known Riemann-Liouville fractional derivative operator defined by
\begin{equation}
D_z^{\lambda} f(z)=\left\{ \begin{array}{ll}
\frac{1}{\Gamma(-\lambda)}\int_0^z(z-t)^{-\lambda-1}f(t)dt& \textrm{$(\Re(\lambda)<0)$}\\
\frac{d^m}{dz^m}D_z^{\lambda-m} f(z)& \textrm{$(m-1\leq\Re(\lambda)<m,\;(m\in\mathbb{N}))$}
\end{array} \right.
\end{equation} 
In \cite [Theorem 3.8]{LU1} Luo et al. proved the following integral representation for the extended Hurwitz-Lerch zeta funtion $\Phi_{\lambda,\mu,\nu}(z,s,a;p,q):$
\begin{equation}
\Phi_{\lambda,\mu,\nu}(z,s,a;p,q)=\frac{1}{\Gamma(s)}\int_0^\infty t^{s-1} e^{-at}{}_2F_1\Big[^{\;\lambda,\;\mu}_{\;\;\nu};ze^{-t};p,q\Big]dt,\;|z|<1,
\end{equation}
$$\left(p,q,a,s>0,\lambda,\mu\in\mathbb{C},\nu\in\mathbb{C}\setminus\mathbb{Z}_0^-\right),$$
where ${}_2F_1\Big[^{\;a,\;b}_{\;\;c};z;p,q\Big]$ is  the extended Gauss hypergeometric function defined by
$${}_2F_1\Big[^{\;a,\;b}_{\;\;c};z;p,q\Big]=\sum_{n=0}^\infty (a)_n\frac{B_{p,q}(b+n,c-b)}{B(b,c-b)}\frac{z^n}{n!},\;|z|<1,$$
$$\Big(\Re(p),\Re(q)\geq0, a, b\in\mathbb{C}, c\in\mathbb{C}\setminus\mathbb{Z}_0^-,\;\Re(c)>\Re(b)>0\Big).$$
When $p=q$ we obtain the extended of the extended of the Gaussian hypergeometric function $F_p$ defined by \cite{CH}:
$${}_2F_1\Big[^{\;a,\;b}_{\;\;c};z;p\Big]=\sum_{n=0}^\infty (a)_n\frac{B_{p}(b+n,c-b)}{B(b,c-b)}\frac{z^n}{n!},\;|z|<1,$$
$$\Big(\Re(p)\geq0, a, b\in\mathbb{C}, c\in\mathbb{C}\setminus\mathbb{Z}_0^-,\;\Re(c)>\Re(b)>0\Big).$$
The Fox-Wright function ${}^p\Psi_q[.]$  with $p$ numerator parameters $\alpha_1,...,\alpha_p$ and $q$ denominator parameters $\beta_1,...,\beta_q$ which are defined by 
\begin{equation}\label{3}
{}_p\Psi_q\Big[_{(\beta_1,B_1),...,(\beta_q,B_q)}^{(\alpha_1,A_1),...,(\alpha_p,A_p)}\Big|z \Big]={}_p\Psi_q\Big[_{(\beta_q,B_q)}^{(\alpha_p,A_p)}\Big|z \Big]=\sum_{k=0}^\infty\frac{\prod_{l=1}^p\Gamma(\alpha_l+kA_l)}{\prod_{j=1}^q\Gamma(\beta_l+kB_l)}\frac{z^k}{k!},
\end{equation}
The defining series in (\ref{3}) converges in the whole complex $z-$plane when
$$\Delta=\sum_{j=1}^q B_j-\sum_{j=1}^p A_j>-1;$$
when $\Delta= 0$, then the series in (\ref{3}) converges for $|z|<\nabla,$ where
$$\nabla=\left(\prod_{j=1}^p A_j^{-A_j}\right)\left(\prod_{j=1}^qB_j^{B_j}\right).$$
If, in the definition (\ref{3}), we set
$$A_1=...=A_p=1\;\;\;\textrm{and}\;\;\;B_1=...=B_q=1,$$
we get the relatively more familiar generalized hypergeometric function ${}_pF_q[.]$ given by
\begin{equation}\label{hyper}
{}_p F_q\left[^{\alpha_1,...,\alpha_p}_{\beta_1,...,\beta_q}\Big|z\right]=\frac{\prod_{j=1}^q\Gamma(\beta_j)}{\prod_{i=1}^p\Gamma(\alpha_i)}{}_p\Psi_q\Big[_{(\beta_1,1),...,(\beta_q,1)}^{(\alpha_1,1),...,(\alpha_p,1)}\Big|z \Big]
\end{equation}

In this paper we consider the $(p,q)-$Mathieu type power series defined by:

\begin{equation}\label{*}
S_{\mu,\nu,\tau,\omega}^{(\alpha,\beta)}(r;\textbf{a};p,q;z)=\sum_{n=1}^\infty \frac{2a_n^\beta(\nu)_n B_{p,q}(\tau+n,\omega-\tau)z^n}{n!B(\tau,\omega-\tau)(a_n^\alpha+r^2)^\mu},
\end{equation}
$$\left(r,\alpha,\beta,\nu>0,\Re(p),\Re(q)\geq0,\;|z|\leq1\right).$$
In particular case when $p=q,$ we define the $p-$Mathieu type power series defined by:
\begin{equation}\label{**}
S_{\mu,\nu,\tau,\omega}^{(\alpha,\beta)}(r;\textbf{a};p;z)=\sum_{n=1}^\infty \frac{2a_n^\beta(\nu)_n B_{p}(\tau+n,\omega-\tau)z^n}{n!B(\tau,\omega-\tau)(a_n^\alpha+r^2)^\mu},
\end{equation}
$$\left(r,\alpha,\beta,\nu>0,\Re(p)\geq0,\;|z|\leq1\right).$$
 The function $S_{\mu,\nu,\tau,\omega}^{(\alpha,\beta)}(r;\textbf{a};p,q;z)$ has many other special cases. We set $p=q=0$ we get
\begin{equation}\label{**}
S_{\mu,\nu,\tau,\omega}^{(\alpha,\beta)}(r;\textbf{a};z)=S_{\mu,\nu,\tau,\omega}^{(\alpha,\beta)}(r;\textbf{a};0,0;z)=\sum_{n=1}^\infty \frac{2a_n^\beta(\nu)_n (\tau)_nz^n}{n!(\omega)_n(a_n^\alpha+r^2)^\mu},
\end{equation}
$$\left(r,\alpha,\beta,\nu,\tau,\omega>0,\;|z|\leq1\right).$$
On the other hand, by letting $\tau=\omega$ in (\ref{**}) we obtain \cite[Eq. 5, p. 974]{ZK}:
\begin{equation}
S_{\mu,\nu}^{(\alpha,\beta)}(r;\textbf{a};z)=S_{\mu,\nu,\tau,\tau}^{(\alpha,\beta)}(r;\textbf{a};z)=\sum_{n=1}^\infty \frac{2a_n^\beta(\nu)_n z^n}{n!(a_n^\alpha+r^2)^\mu},
\end{equation}
$$\left(r,\alpha,\beta,\nu>0,\;|z|\leq1\right).$$
Furthermore, the special cases when $\nu=z=1$ we get the  generalized Mathieu series (\ref{;,}).

The contents of our paper is organized as follows. In section 2, we present new integral representation for the $(p,q)-$Mathieu series. In particular, we  derive the Mellin-Barnes type integral representations
for $(p,q)-$Mathieu series $S_{\mu,\nu,\tau,\omega}^{(2,1)}\Big(r;\{k\}_{k=0}^\infty;p,q;-z\Big).$ As applications, In Section 3, we introduce the $(p,q)-$Mittag-Leffler functions and we derive some relationships between thus two special functions, in particular we derive new series representations for the $(p,q)-$Mathieu series.  Relationships between the $(p,q)-$ and generalized Mathieu series are proved and two Tur\'an type inequalities are established.

\section{\bf Integral representation for the $(p,q)-$Mathieu types series}

In the course of our investigation, one of the main tools is the following result providing the integral representation for the $(p,q)-$Mathieu types power series $S_{\mu,\nu,\tau,\omega}^{(\alpha,\beta)}(r;\{k^\gamma\}_{k=0}^\infty;p,q;z).$

\begin{theorem}\label{T1}Let $r,\alpha,\beta,\nu,\mu,\tau,\omega>0,\;\Re(p),\Re(q)\geq0$ such that $\gamma(\mu\alpha-\beta)>0.$ Then $(p,q)-$Mathieu types power series $S_{\mu,\nu,\tau,\omega}^{(\alpha,\beta)}(r;\{k^\gamma\}_{k=0}^\infty;p,q;z)$ possesses the integral representation given by:
$$S_{\mu,\nu,\tau,\omega}^{(\alpha,\beta)}(r;\{k^\gamma\}_{k=0}^\infty;p,q;z)=$$
\begin{equation}\label{a}
=\frac{2\nu \tau z}{\omega\Gamma(\mu)}\int_0^\infty t^{\gamma[(\mu\alpha-\beta]}e^{-t}{}_2F_1\Big[^{\;\nu+1,\tau+1}_{\;\omega+1};ze^{-t};p,q\Big]{}_1\Psi_1\Big[^{\;\;\;\;(\mu,1)}_{(\gamma(\mu\alpha-\beta)+1,\gamma\alpha)}\Big|-r^2t^{\gamma\alpha}\Big]dt.
\end{equation}
\end{theorem}
\begin{proof}By using the definition (\ref{*}), we can write the extended Mathieu types series $S_{\mu,\nu,\tau,\omega}^{(\alpha,\beta)}(r;\textbf{a};p,q;z)$ in the following form:
\begin{equation}\label{mmm}
S_{\mu,\nu,\tau,\omega}^{(\alpha,\beta)}(r;\textbf{a};p,q;z)=2\sum_{m=0}^\infty\binom{\mu+m-1}{m}(-r^2)^m\sum_{n=1}^\infty\frac{(\nu)_n}{a_n^{(\mu+m)\alpha-\beta}}\frac{B_{p,q}(\tau+n,\omega-\tau)}{B(\tau,\omega-\tau)}\frac{z^n}{n!}.
\end{equation}
Therefore,
$$S_{\mu,\nu,\tau,\omega}^{(\alpha,\beta)}\Big(r;\{k^\gamma\}_{k=0}^\infty;p,q;z\Big)=$$
\begin{equation*}
\begin{split}
\;\;\;\;&=2z\sum_{m=0}^\infty\binom{\mu+m-1}{m}(-r^2)^m\sum_{n=0}^\infty\frac{(\nu)_{n+1}}{(n+1)!}\frac{B_{p,q}(\tau+1+n,\omega-\tau)}{B(\tau,\omega-\tau)}\frac{z^n}{(n+1)^{\gamma((\mu+m)\alpha-\beta})}\\
\;\;\;\;&=2\nu z\sum_{m=0}^\infty\binom{\mu+m-1}{m}(-r^2)^m\sum_{n=0}^\infty\frac{(\nu+1)_{n}}{n!}\frac{B_{p,q}(\tau+1+n,\omega-\tau)}{B(\tau,\omega-\tau)}\frac{z^n}{(n+1)^{\gamma((\mu+m)\alpha-\beta)+1}}\\
\;\;\;\;&=\frac{2\nu zB(\tau+1,\omega-\tau)}{B(\tau,\omega-\tau)} \sum_{m=0}^\infty\binom{\mu+m-1}{m}(-r^2)^m\sum_{n=0}^\infty\frac{(\nu+1)_{n}B_{p,q}(\tau+1+n,\omega-\tau)z^n}{n!B(\tau+1,\omega-\tau)(n+1)^{\gamma((\mu+m)\alpha-\beta)+1}}\\
\;\;\;\;&=\frac{2\nu z B(\tau+1,\omega-\tau)}{B(\tau,\omega-\tau)} \sum_{m=0}^\infty\binom{\mu+m-1}{m}(-r^2)^m \Phi_{\nu+1,\tau+1,\omega+1}(z,\gamma[(\mu+m)\alpha-\beta]+1,1;p,q)\\
\;\;\;\;\;\;\;\;\;\;\;\;&=\frac{2\nu \tau z}{\omega}\int_0^\infty t^{\gamma[(\mu\alpha-\beta]}e^{-t}{}_2F_1\left[^{\;\nu+1,\tau+1}_{\;\omega+1};ze^{-t};p,q\right]\left(\sum_{m=0}^\infty\frac{\binom{\mu+m-1}{m}(-r^2t^{\gamma\alpha})^m }{\Gamma(\gamma[(\mu+m)\alpha-\beta]+1)}\right)dt\\
\;\;\;\;\;\;\;\;&=\frac{2\nu \tau z}{\omega\Gamma(\mu)}\int_0^\infty t^{\gamma[(\mu\alpha-\beta]}e^{-t}{}_2F_1\Big[^{\;\nu+1,\tau+1}_{\;\omega+1};ze^{-t};p,q\Big]{}_1\Psi_1\Big[^{\;\;\;\;(\mu,1)}_{(\gamma(\mu\alpha-\beta)+1,\gamma\alpha)}\Big|-r^2t^{\gamma\alpha}\Big]dt.
\end{split}
\end{equation*}
This completes the proof of Theorem \ref{T1}.
\end{proof}

Now, in ths case $p=q$, Theorem \ref{T1}  reduces to the following corollary.

\begin{coro}Let $r,\alpha,\beta,\nu,\mu,\tau,\omega>0,\;\Re(p)\geq0$ such that $\gamma(\mu\alpha-\beta)>0.$ Then $(p,q)-$Mathieu types power series $S_{\mu,\nu,\tau,\omega}^{(\alpha,\beta)}(r;\{k^\gamma\}_{k=0}^\infty;p;z)$ possesses the integral representation given by:
\begin{equation}\label{a}
S_{\mu,\nu,\tau,\omega}^{(\alpha,\beta)}(r;\{k^\gamma\}_{k=0}^\infty;p;z)=\frac{2\nu \tau z}{\omega\Gamma(\mu)}\int_0^\infty t^{\gamma[(\mu\alpha-\beta]}e^{-t}{}_2F_1\Big[^{\;\nu+1,\tau+1}_{\;\omega+1};ze^{-t};p\Big]{}_1\Psi_1\Big[^{\;\;\;\;(\mu,1)}_{(\gamma(\mu\alpha-\beta)+1,\gamma\alpha)}\Big|-r^2t^{\gamma\alpha}\Big]dt.
\end{equation}
\end{coro}
\begin{remark}
1. By letting $p=q=0$ in (\ref{a}), we deduce that the function $S_{\mu,\nu,\tau,\omega}^{(\alpha,\beta)}(r;\{k^\gamma\}_{k=0}^\infty;z)$ possesses the following integral representation:
\begin{equation}\label{b}
S_{\mu,\nu,\tau,\omega}^{(\alpha,\beta)}(r;\{k^\gamma\}_{k=0}^\infty;z)=\frac{2\nu \tau z}{\omega\Gamma(\mu)}\int_0^\infty t^{\gamma[(\mu\alpha-\beta]}e^{-t}{}_2F_1\Big[^{\;\nu+1,\tau+1}_{\;\omega+1};ze^{-t}\Big]{}_1\Psi_1\Big[^{\;\;\;\;(\mu,1)}_{(\gamma(\mu\alpha-\beta)+1,\gamma\alpha)}\Big|-r^2t^{\gamma\alpha}\Big]dt.
\end{equation}
2. Setting $\tau=\omega$ in (\ref{b}) and using the fact that 
$${}_2F_1\Big[^{\;a,\;b}_{\;\;b};z\Big]=(1-z)^{-a}$$
we obtain the following integral representation for the function $S_{\mu,\nu}^{(\alpha,\beta)}(r;\{k^\gamma\}_{k=0}^\infty;z)$\cite[Theorem 1, Eq. 8]{ZK}
\begin{equation}
S_{\mu,\nu}^{(\alpha,\beta)}(r;\{k^\gamma\}_{k=0}^\infty;z)=\frac{2\nu z}{\Gamma(\mu)}\int_0^\infty \frac{t^{\gamma[(\mu\alpha-\beta]}e^{-t}}{(1-ze^{-t})^{\nu+1}}{}_1\Psi_1\Big[^{\;\;\;\;(\mu,1)}_{(\gamma(\mu\alpha-\beta)+1,\gamma\alpha)}\Big|-r^2t^{\gamma\alpha}\Big]dt.
\end{equation}
\end{remark}

In the next Theorem we present the Mellin-Barnes integral representation for the alternating Mathieu-series $S_{\mu,\nu,\tau,\omega}^{(2,1)}\Big(r;\{k\}_{k=0}^\infty;p,q;-z\Big).$

\begin{theorem}\label{TTT222222222222}Let $r,\nu,\mu,\tau,\omega>0,\;\Re(p),\Re(q)\geq0.$ Then the following integral representation
$$S_{\mu,\nu,\tau,\omega}^{(2,1)}\Big(r;\{k\}_{k=0}^\infty;p,q;-z\Big)=$$
\begin{equation}
=-\frac{z}{i\pi\Gamma(\nu)}\int_{c-i\infty}^{c+i\infty}\frac{\Gamma(s)\Gamma(\nu-s+1)B_{p,q}(\tau-s+1,\omega-\tau)\left[\Gamma(-s+ir+1)\Gamma(-s-ir+1)\right]^\mu}{B(\tau,\omega-\tau)\left[\Gamma(-s+ir+2)\Gamma(-s-ir+2)\right]^\mu}z^{-s}ds,
\end{equation}
holds true for all $|\arg(-z)|<\pi$.
\end{theorem}
\begin{proof}The contour of integration extends from $c-i\infty$ to $c+i\infty,$ such that all the poles of the Gamma function $\Gamma(\nu-s+1)$ at the points $s=k+\nu+1,\;k\in\mathbb{N}$ are separated from the poles of the gamma function $\Gamma(s)$ at the points $s=-k,\; k\in\mathbb{N}.$ Suppose that the the poles of the integrand are simple and using the fact that 
$$ \textrm{res}[\Gamma,-k] =\lim_{s\longrightarrow -k}(s+k)\Gamma(s)=\frac{(-1)^k}{k!},$$
we find that
$$\frac{z}{i\pi\Gamma(\nu)}\int_{c-i\infty}^{c+i\infty}\frac{\Gamma(s)\Gamma(\nu-s+1)B_{p,q}(\tau-s+1,\omega-\tau)\left[\Gamma(-s+ir+1)\Gamma(-s-ir+1)\right]^\mu}{B(\tau,\omega-\tau)\left[\Gamma(-s+ir+2)\Gamma(-s-ir+2)\right]^\mu}z^{-s}ds$$
\begin{equation*}
\begin{split}
&=\frac{2z}{\Gamma(\nu)}\sum_{k=0}^\infty \lim_{s\longrightarrow -k}\frac{(s+k)\Gamma(s)B_{p,q}(\tau-s+1,\omega-\tau)\Gamma(\nu-s+1)\left[\Gamma(-s+ir+1)\Gamma(-s-ir+1)\right]^\mu}{B(\tau,\omega-\tau)\left[\Gamma(-s+ir+2)\Gamma(-s-ir+2)\right]^\mu}\\
&=\frac{2z}{\Gamma(\nu)}\sum_{k=0}^\infty \frac{(-1)^k}{k!}\frac{B_{p,q}(\tau+k+1,\omega-\tau)\Gamma(\nu+k+1)}{B(\tau,\omega-\tau)((k+1)^2+r^2)^\mu}z^k\\
&=-2\sum_{k=1}^\infty\frac{(\nu)_k k B_{p,q}(\tau+k,\omega-\tau)}{B(\tau,\omega-\tau)(k^2+r^2)^\mu}\frac{(-z)^k}{k!}\\
&=-S_{\mu,\nu,\tau,\omega}^{(2,1)}\Big(r;\{k\}_{k=0}^\infty;p,q;-z\Big).
\end{split}
\end{equation*}
This completes the proof of Theorem \ref{TTT222222222222}
\end{proof}
\begin{coro}\label{cccc}
Let $r,\nu,\mu,\tau,\omega>0,\;\Re(p)\geq0.$ Then the following integral representation
$$S_{\mu,\nu,\tau,\omega}^{(2,1)}\Big(r;\{k\}_{k=0}^\infty;p;-z\Big)=$$
\begin{equation}
=-\frac{z}{i\pi\Gamma(\nu)}\int_{c-i\infty}^{c+i\infty}\frac{\Gamma(s)\Gamma(\nu-s+1)B_{p}(\tau-s+1,\omega-\tau)\left[\Gamma(-s+ir+1)\Gamma(-s-ir+1)\right]^\mu}{B(\tau,\omega-\tau)\left[\Gamma(-s+ir+2)\Gamma(-s-ir+2)\right]^\mu}z^{-s}ds,
\end{equation}
holds true for all $|\arg(-z)|<\pi$.
\end{coro}
\begin{remark}If we set $p=0$  in Corollary \ref{cccc}, then we get the Mellin-Barnes representation of the function $S_{\mu,\nu,\tau,\omega}^{(2,1)}\Big(r;\{k\}_{k=0}^\infty;-z\Big):$
$$S_{\mu,\nu,\tau,\omega}^{(2,1)}\Big(r;\{k\}_{k=0}^\infty;-z\Big)=$$
\begin{equation}
=-\frac{z\Gamma(\omega)}{i\pi\Gamma(\nu)\Gamma(\tau)}\int_{c-i\infty}^{c+i\infty}\frac{\Gamma(s)\Gamma(\nu-s+1)\Gamma(\tau-s+1)\left[\Gamma(-s+ir+1)\Gamma(-s-ir+1)\right]^\mu}{\Gamma(\omega-s+1)\left[\Gamma(-s+ir+2)\Gamma(-s-ir+2)\right]^\mu}z^{-s}ds
\end{equation}
In particular, for $\tau=\omega$ we get
$$S_{\mu,\nu}^{(2,1)}\Big(r;\{k\}_{k=0}^\infty;-z\Big)=
-\frac{z}{i\pi\Gamma(\nu)}\int_{c-i\infty}^{c+i\infty}\frac{\Gamma(s)\Gamma(\nu-s+1)\left[\Gamma(-s+ir+1)\Gamma(-s-ir+1)\right]^\mu}{\left[\Gamma(-s+ir+2)\Gamma(-s-ir+2)\right]^\mu}z^{-s}ds.$$
Moreover, if we set $\mu=2$ and $\nu=1$ in the above equation we get the Mellin-Barnes for the  alternating Mathieu-series  proved by Saxena el al. \cite[Theorem 3.1]{SA}.
\end{remark}
\section{\bf Applications}

In our first application in this section we present the relationships between the $(p,q)-$Mathieu-type	series $S_{2,\nu,\tau,\omega}^{(2,1)}\Big(r;\{k\}_{k=0}^\infty;p,q;z\Big)$ and  the Rieman-Liouvile operator.

\subsection{Relationships with $(p,q)-$Mathieu-type	series and the Rieman-Liouvile operator}
Our first main application is asserted by the following Theorem.

\begin{theorem} Let $r,\mu,\tau,\omega>0,\;\Re(p),\Re(q)\geq0$ and $0\leq\nu<1$. Then 
\begin{equation}
S_{2,\nu,\tau,\omega}^{(2,1)}\Big(r;\{k\}_{k=0}^\infty;p,q;z\Big)=\frac{1}{2ir\Gamma(\nu)}\left\{D_z^{\nu-1}\left(z^{\nu-1}\Phi_{\tau,\omega}(z,2,-ir;p,q)\right)-D_z^{\nu-1}\left(z^{\nu-1}\Phi_{\tau,\omega}(z,2,ir;p,q)\right)\right\}.
\end{equation}
\end{theorem}
\begin{proof}By using the definition of the $(p,q)-$Mathieu-type	series, we can write the Mathieu-type series\\ $S_{2,\nu,\tau,\omega}^{(2,1)}\Big(r;\{k\}_{k=0}^\infty;p,q;z\Big)$ in the following form:
\begin{equation}
S_{2,\nu,\tau,\omega}^{(2,1)}\Big(r;\{k\}_{k=0}^\infty;p,q;z\Big)=\frac{1}{2ir}\left[\Phi_{\nu,\tau,\omega}(z,2,-ir;p,q)-\Phi_{\nu,\tau,\omega}(z,2,ir;p,q)\right].
\end{equation}
Combining the above equation with (\ref{ZZ}), we get the desired result.
 \end{proof}
\subsection{Relationships with $(p,q)-$Mittag-Leffler function and $(p,q)-$Mathieu-type	series} In this section, we introduce the definition of the $(p,q)-$Mittag-Leffler function and we establish an integral representation for this function and we present some relationships with the  $(p,q)-$Mathieu-type	series.
For $\lambda,\tau,\omega,\theta,\sigma, \delta>0$  and $\Re(p),\Re(q)\geq0$ we define the $(p,q)-$Mittag-Leffler function by
\begin{equation}\label{def1}
E_{\delta,\theta,\sigma;p,q}^{(\lambda,\tau,\omega)}(z)=\sum_{k=0}^\infty \frac{(\lambda)_k}{\left[\Gamma(\theta k+\sigma)\right]^\delta}\frac{B_{p,q}(\tau+k,\omega-\tau)}{B(\tau,\omega-\tau)}\frac{z^k}{k!}, z\in\mathbb{C}.
\end{equation}
In the case $p=q$ we define the $p-$Mittag-Leffler function by
\begin{equation}
E_{\delta,\theta,\sigma;p}^{(\lambda,\tau,\omega)}(z)=\sum_{k=0}^\infty \frac{(\lambda)_k}{\left[\Gamma(\theta k+\sigma)\right]^\delta}\frac{B_{p}(\tau+k,\omega-\tau)}{B(\tau,\omega-\tau)}\frac{z^k}{k!}, z\in\mathbb{C},
\end{equation}
whose special case when $p=0$ reduces to the generalized Mittag-Leffler function, introduced by Tomovski and Mehrez in \cite{ZK}
\begin{equation}
E_{\delta,\theta,\sigma}^{(\lambda)}(z)=\sum_{k=0}^\infty \frac{(\lambda)_k}{\left[\Gamma(\theta k+\sigma)\right]^\delta}\frac{z^k}{k!}, z\in\mathbb{C}.
\end{equation}
For $\lambda=1$ the above series was introduced by S. Gerhold \cite{GG}.
\begin{lemma}\label{LLLL}For $\tau,\omega,\theta,\sigma,\delta>0$  and $\Re(p),\Re(q)\geq0.$ Then we have
\begin{equation}\label{55555}
E_{\delta,\theta,\sigma+\theta;p,q}^{(1,\tau+1,\omega+1)}(z)=\frac{\omega}{z\tau}\left[E_{\delta,\theta,\sigma;p,q}^{(1,\tau,\omega)}(z)-\frac{B_{p,q}(\tau,\omega-\tau)}{[\Gamma(\sigma)]^{\delta}B(\tau,\omega-\tau)}\right].
\end{equation}
\end{lemma}
\begin{proof}By computation, we get
\begin{equation*}
\begin{split}
E_{\delta,\theta,\sigma+\theta;p,q}^{(1,\tau+1,\omega+1)}(z)&=\sum_{k=0}^\infty \frac{B_{p,q}(\tau+k+1,\omega-\tau)z^k}{B(\tau+1,\omega-\tau)\left[\Gamma(\theta k+\sigma+\theta)\right]^\delta}\\
&=\frac{B(\tau,\omega-\tau)}{zB(\tau+1,\omega-\tau)}\sum_{k=1}^\infty \frac{B_{p,q}(\tau+k,\omega-\tau)z^k}{B(\tau,\omega-\tau)\left[\Gamma(\theta k+\sigma)\right]^\delta}\\
&=\frac{\omega}{z\tau}\left[E_{\delta,\theta,\sigma;p,q}^{(1,\tau,\omega)}(z)-\frac{B_{p,q}(\tau,\omega-\tau)}{[\Gamma(\sigma)]^{\delta}B(\tau,\omega-\tau)}\right].
\end{split}
\end{equation*}
The proof of Lemma \ref{LLLL} is completes.
\end{proof}
\begin{theorem}\label{8888}Let  $\lambda,\tau,\omega,\theta,\sigma>0, \delta\in\mathbb{N}$  and $\Re(p),\Re(q)\geq0.$Then the  $(p,q)-$Mathieu-type	series  admits the following series representation:
$$S_{\mu,\nu,\tau,\omega}^{(\alpha,\beta)}\Big(r;\{[\Gamma(\theta k+\sigma)]^\gamma\}_{k=0}^\infty;p,q;z\Big)$$
\begin{equation}\label{--}
\begin{split}
\qquad\qquad\qquad\qquad=2\sum_{m=0}^\infty\binom{\mu+m-1}{m}(-r^2)^m \left[E_{\gamma[(\mu+m)\alpha-\beta],\theta,\sigma;p,q}^{(\nu,\tau,\omega)}(z)-\frac{B_{p,q}(\tau,\omega-\tau)}{[\Gamma(\sigma)]^{\gamma[(\mu+m)\alpha-\beta]}B(\tau,\omega-\tau)}\right].
\end{split}
\end{equation}
Moreover, the following series representation
\begin{equation}\label{;;}
S_{\mu,1 ,\tau,\omega}^{(\alpha,\beta)}\Big(r;\{[\Gamma(\theta k+\sigma)]^\gamma\}_{k=0}^\infty;p,q;z\Big)=\frac{2z\tau}{\omega}\sum_{m=0}^\infty\binom{\mu+m-1}{m}(-r^2)^m E_{\gamma[(\mu+m)\alpha-\beta],\theta,\sigma+\theta;p,q}^{(1,\tau+1,\omega+1)}(z).
\end{equation}
holds true.
\end{theorem}
\begin{proof} In view of the definition of the $(p,q)-$Mittag-Leffler function (\ref{def1}) and the equation (\ref{mmm}) we obtain (\ref{--}). Finally, combining the equation (\ref{--}) with the following relation (\ref{55555}) we obtain the formula (\ref{;;}).
\end{proof}

Taking in (\ref{--}) the values $\theta=\sigma=1$ we obtain the following representation:

\begin{coro}Let  $\lambda,\tau,\omega,\theta,\sigma>0, \delta\in\mathbb{N}$  and $\Re(p),\Re(q)\geq0.$Then the  $(p,q)-$Mathieu-type	series  admits the following series representations:
\begin{equation}
\begin{split}
S_{\mu,\nu,\tau,\omega}^{(\alpha,\beta)}\Big(r;\{(k!)^\gamma\}_{k=0}^\infty;p,q;z\Big)=2\sum_{m=0}^\infty\binom{\mu+m-1}{m}(-r^2)^m \left[E_{\gamma[(\mu+m)\alpha-\beta],1,1;p,q}^{(\nu,\tau,\omega)}(z)-\frac{B_{p,q}(\tau,\omega-\tau)}{B(\tau,\omega-\tau)}\right],
\end{split}
\end{equation}
and 
\begin{equation}
S_{\mu,1 ,\tau,\omega}^{(\alpha,\beta)}\Big(r;\{k!^\gamma\}_{k=0}^\infty;p,q;z\Big)=\frac{2z\tau}{\omega}\sum_{m=0}^\infty\binom{\mu+m-1}{m}(-r^2)^m E_{\gamma[(\mu+m)\alpha-\beta],1,2;p,q}^{(1,\tau+1,\omega+1)}(z).
\end{equation}
\end{coro}

\begin{lemma}\label{L1}For $\lambda,\tau,\omega,\theta,\sigma>0, \delta\in\mathbb{N}$  and $\Re(p),\Re(q)\geq0.$  Then the the $(p,q)-$Mittag-Leffler function $E_{\delta,\theta,\sigma;p,q}^{(\lambda,\tau,\omega)}(z)$ possesses the following integral representation:
\begin{equation}\label{int1}
E_{\delta,\theta,\sigma;p,q}^{(\lambda,\tau,\omega)}(z)=\frac{1}{B(\tau,\omega-\tau)}\int_0^1 t^{\tau-1}(1-t)^{\omega-\tau-1}E_{p,q}(t)E_{\delta,\theta,\sigma}^{(\lambda)}(zt)dt,
\end{equation}
holds true.
\end{lemma}
\begin{proof}By using the definition of the $(p,q)-$Beta function we get
\begin{equation*}
\begin{split}
\int_0^1 t^{\tau-1}(1-t)^{\omega-\tau-1}E_{p,q}(t)E_{\delta,\theta,\sigma}^{(\lambda)}(zt)dt&=\int_0^1 t^{\tau-1}(1-t)^{\omega-\tau-1}E_{p,q}(t)\left(\sum_{k=0}^\infty\frac{(\lambda)_k(zt)^k}{\left[\Gamma(\theta k+\sigma)\right]^\delta k!}\right)dt\\
&=\sum_{k=0}^\infty\frac{(\lambda)_k z^k}{\left[\Gamma(\theta k+\sigma)\right]^\delta k!}\int_0^1 t^{\tau+k-1}(1-t)^{\omega-\tau-1}E_{p,q}(t)dt\\
&=B(\tau,\omega-\tau)\sum_{k=0}^\infty \frac{(\lambda)_k}{\left[\Gamma(\theta k+\sigma)\right]^\delta}\frac{B_{p,q}(\tau+k,\omega-\tau)}{B(\tau,\omega-\tau)}\frac{z^k}{k!}\\
&=B(\tau,\omega-\tau)E_{\delta,\theta,\sigma;p,q}^{(\lambda,\tau,\omega)}(z).
\end{split}
\end{equation*}
The proof of Lemma \ref{L1} is completes.
\end{proof}
\begin{theorem}\label{T5}For $\lambda,\tau,\omega,\theta,\sigma>0, \delta\in\mathbb{N}$  and $\Re(p),\Re(q)\geq0.$  Then the following integral representation 
\begin{equation}\label{HHHHH}
\begin{split}
S_{\mu,\nu,\tau,\omega}^{(\alpha,\beta)}\Big(r;\{[\Gamma(\theta k+\sigma)]^\gamma\}_{k=0}^\infty;p,q;z\Big)&=\frac{2}{B(\tau,\omega-\tau)}\int_0^1 t^{\tau-1}(1-t)^{\omega-\tau-1}E_{p,q}(t)S_{\mu,\nu}^{(\alpha,\beta)}\Big(r;\{[\Gamma(\theta k+\sigma)]^\gamma\}_{k=0}^\infty;zt\Big)dt\\
&-\frac{2B_{p,q}(\tau,\omega-\tau)}{B(\tau,\omega-\tau)}.\frac{[\Gamma(\sigma)]^\beta}{(\tau^2+[\Gamma(\sigma)]^\alpha)^\mu}
\end{split}
\end{equation}
holds true for all $|z|<1.$ Moreover, the following integral representation
\begin{equation}\label{HHHHHH}
\begin{split}
S_{\mu,1,\tau,\omega}^{(\alpha,\beta)}\Big(r;\{[\Gamma(\theta k+\sigma)]^\gamma\}_{k=0}^\infty;p,q;z\Big)=
\end{split}
\end{equation}
$$\;\;\;\;\;\;\;\;\;\;\;\;\;\;=\frac{2z\tau}{\omega B(\tau+1,\omega-\tau)}\int_0^1 t^{\tau-1}(1-t)^{\omega-\tau-1}E_{p,q}(t)S_{\mu,1}^{(\alpha,\beta)}\Big(r;\{[\Gamma(\theta k+\sigma)]^\gamma\}_{k=0}^\infty;zt\Big)dt,$$
holds true for all $|z|<1.$ 
\end{theorem}
\begin{proof}By means of Lemma \ref{L1} and the integral representation (\ref{--}), we get
$$S_{\mu,\nu,\tau,\omega}^{(\alpha,\beta)}\Big(r;\{[\Gamma(\theta k+\sigma)]^\gamma\}_{k=0}^\infty;p,q;z\Big)=$$
\begin{equation*}
\begin{split}
&=2\sum_{m=0}^\infty\binom{\mu+m-1}{m}(-r^2)^m \left[\frac{1}{B(\tau,\omega-\tau)}\int_0^1 t^{\tau-1}(1-t)^{\omega-\tau-1}E_{p,q}(t)E_{\gamma[(\mu+m)\alpha-\beta],\theta,\sigma}^{(\nu)}(zt)dt\right]\\
&-\frac{2B_{p,q}(\tau,\omega-\tau)}{[\Gamma(\sigma)]^{\mu\alpha-\beta}B(\tau,\omega-\tau)}\sum_{m=0}^\infty\binom{\mu+m-1}{m} \left(\frac{-\tau^2}{[\Gamma(\sigma)]^\alpha}\right)^m\\
&=\frac{2}{B(\tau,\omega-\tau)}\int_0^1 t^{\tau-1}(1-t)^{\omega-\tau-1}E_{p,q}(t)\left(\sum_{m=0}^\infty\binom{\mu+m-1}{m}(-r^2)^m E_{\gamma[(\mu+m)\alpha-\beta],\theta,\sigma}^{(\nu)}(zt)\right)dt\\
&-\frac{2B_{p,q}(\tau,\omega-\tau)}{[\Gamma(\sigma)]^{\mu\alpha-\beta}B(\tau,\omega-\tau)}.\frac{1}{(1+\frac{r^2}{[\Gamma(\sigma)]^\alpha})^\mu}\\
&=\frac{2}{B(\tau,\omega-\tau)}\int_0^1 t^{\tau-1}(1-t)^{\omega-\tau-1}E_{p,q}(t)\sum_{k=0}^\infty\frac{(\nu)_k}{[\Gamma(\theta k+\sigma)]^{\gamma(\mu\alpha-\beta)}}\left(\sum_{m=0}^\infty\frac{\binom{\mu+m-1}{m}(-r^2)^m}{[\Gamma(\theta k+\sigma)]^{\gamma m\alpha}}\right)\frac{(zt)^k}{k!}dt\\
&-\frac{2B_{p,q}(\tau,\omega-\tau)}{B(\tau,\omega-\tau)}.\frac{[\Gamma(\sigma)]^\beta}{(r^2+[\Gamma(\sigma)]^\alpha)^\mu}\\
\end{split}
\end{equation*}
\begin{equation*}
\begin{split}
&=\frac{2}{B(\tau,\omega-\tau)}\int_0^1 t^{\tau-1}(1-t)^{\omega-\tau-1}E_{p,q}(t)\sum_{k=0}^\infty\frac{(\nu)_k}{[\Gamma(\theta k+\sigma)]^{\gamma(\mu\alpha-\beta)}}\left(1+\frac{r^2}{[\Gamma(\theta k+\sigma)]^{\gamma\alpha}}\right)^{-\mu}\frac{(zt)^k}{k!}dt\\
&-\frac{2B_{p,q}(\tau,\omega-\tau)}{B(\tau,\omega-\tau)}.\frac{[\Gamma(\sigma)]^\beta}{(r^2+[\Gamma(\sigma)]^\alpha)^\mu}\\
&=\frac{2}{B(\tau,\omega-\tau)}\int_0^1 t^{\tau-1}(1-t)^{\omega-\tau-1}E_{p,q}(t)\sum_{k=0}^\infty\frac{(\nu)_k[\Gamma(\theta k+\sigma)]^{\gamma\beta}}{\left([\Gamma(\theta k+\sigma)]^{\gamma\alpha}+r^2\right)^\mu}\frac{(zt)^k}{k!}dt\\
&-\frac{2B_{p,q}(\tau,\omega-\tau)}{B(\tau,\omega-\tau)}.\frac{[\Gamma(\sigma)]^\beta}{(\tau^2+[\Gamma(\sigma)]^\alpha)^\mu}\\
&=\frac{2}{B(\tau,\omega-\tau)}\int_0^1 t^{\tau-1}(1-t)^{\omega-\tau-1}E_{p,q}(t)S_{\mu,\nu}^{(\alpha,\beta)}\Big(r;\{[\Gamma(\theta k+\sigma)]^\gamma\}_{k=0}^\infty;zt\Big)dt\\&-\frac{2B_{p,q}(\tau,\omega-\tau)}{B(\tau,\omega-\tau)}.\frac{[\Gamma(\sigma)]^\beta}{(r^2+[\Gamma(\sigma)]^\alpha)^\mu},
\end{split}
\end{equation*}
which evidently completes the proof of the representation (\ref{HHHHH}). Finally, combining (\ref{int1}) and (\ref{;;}) and repeating the same calculations as above we get (\ref{HHHHHH}). The proof of Theorem \ref{T5} is completes.
\end{proof}
\subsection{Tur\'an type inequalities for the $(p,q)-$Mathieu-type	series }
\begin{theorem}Let $r,\alpha,\beta,\nu,\mu,\tau,\omega>0,\;p,q\geq0.$ Then the following assertions are true:
1. The $(p,q)-$Mathieu-type	series considered as a function in $p$ $( $or $q)$ is completely monotonic and log-convex on $(0,\infty).$ Furthermore, the following Tur\'an type inequality
\begin{equation}\label{TURAN}
\left[S_{\mu,\nu,\tau,\omega}^{(\alpha,\beta)}(r;\textbf{a};p+1,q;z)\right]^2\leq S_{\mu,\nu,\tau,\omega}^{(\alpha,\beta)}(r;\textbf{a};p,q;z)S_{\mu,\nu,\tau,\omega}^{(\alpha,\beta)}(r;\textbf{a};p+2,q;z)
\end{equation}
holds true for all $z\in (0,1).$\\
2. Assume that $r^2+\textbf{a}\geq1.$ Then the $(p,q)-$Mathieu-type	series considered as a function in $\mu$  is completely monotonic and log-convex on $(0,\infty).$ Furthermore, the following Tur\'an type inequality
\begin{equation}\label{TURAN1}
\left[S_{\mu+1,\nu,\tau,\omega}^{(\alpha,\beta)}(r;\textbf{a};p,q;z)\right]^2\leq S_{\mu,\nu,\tau,\omega}^{(\alpha,\beta)}(r;\textbf{a};p,q;z)S_{\mu+2,\nu,\tau,\omega}^{(\alpha,\beta)}(r;\textbf{a};p,q;z)
\end{equation}
holds true for all $z\in (0,1)$ such that $r^2+\textbf{a}\geq1.$
\end{theorem}
\begin{proof}1. In \cite[Corollary 2.7]{LU1}, the authors proved that the extended beta function $p(\;\textrm{or}\;q)\mapsto B_{p,q} (x, y)$ is completely monotonic function on $(0,\infty)$ and using the fact that sums of completely monotonic functions are completely monotonic too, we deduce that  the $p (\;\textrm {or}\;q)\mapsto S_{\mu,\nu,\tau,\omega}^{(\alpha,\beta)}(r;\textbf{a};p,q;z)$ is completely monotonic and log-convex on $(0,\infty),$  since every completely monotonic function is log-convex ( see \cite[p.167]{WI}. Thus, for all $p_1,p_2>0,$ and $t\in[0,1],$ we obtain
\begin{equation*}
S_{\mu,\nu,\tau,\omega}^{(\alpha,\beta)}(r;\textbf{a};tp_1+(1-t)p_2,q;z)\leq\left[S_{\mu,\nu,\tau,\omega}^{(\alpha,\beta)}(r;\textbf{a};p_1,q;z)\right]^t \left[S_{\mu,\nu,\tau,\omega}^{(\alpha,\beta)}(r;\textbf{a};p_2,q;z)\right]^{1-t}.
\end{equation*}
Letting $t=\frac{1}{2},\;p_1=p$ and $p_2=p+2$ in the above inequality we get the Tur\'an type inequality (\ref{TURAN}).\\
2. We note that the function $\mu\mapsto (r^2+\textbf{a})^{-\mu}$ is completely monotonic on $(0,\infty)$ such that $r^2+\textbf{a}\geq1,$ and consequently the function $\mu\mapsto S_{\mu,\nu,\tau,\omega}^{(\alpha,\beta)}(r;\textbf{a};p,q;z)$ is completely monotonic and log-convex on $(0,\infty).$
\end{proof}
\begin{remark}The condition $r^2+\textbf{a}$ is not necessary for proved the Tur\'an type inequality (\ref{TURAN1}), a similar proof of the Theorem 2 in \cite{SKZ}, we obtain (\ref{TURAN1}).
\end{remark}

\end{document}